\documentclass[11pt,reqno]{amsart}

\usepackage{microtype}
\usepackage[letterpaper,left=1in,right=1in,top=1in,bottom=1in]{geometry}
\usepackage{xcolor}
\usepackage{mathtools}
\usepackage{amssymb}
\usepackage{mathrsfs}
\usepackage{aliascnt}
\usepackage[colorlinks=true,linkcolor=blue!55!black,citecolor=blue!55!black,urlcolor=blue!55!black]{hyperref}
\usepackage[nameinlink,capitalize,noabbrev]{cleveref}

\allowdisplaybreaks
\numberwithin{equation}{section}

\newtheorem{theorem}{Theorem}[section]
\newaliascnt{proposition}{theorem}
\newtheorem{proposition}[proposition]{Proposition}
\aliascntresetthe{proposition}
\newaliascnt{lemma}{theorem}
\newtheorem{lemma}[lemma]{Lemma}
\aliascntresetthe{lemma}
\newaliascnt{corollary}{theorem}
\newtheorem{corollary}[corollary]{Corollary}
\aliascntresetthe{corollary}
\theoremstyle{remark}
\newaliascnt{remark}{theorem}
\newtheorem{remark}[remark]{Remark}
\aliascntresetthe{remark}

\crefname{theorem}{Theorem}{Theorems}
\crefname{proposition}{Proposition}{Propositions}
\crefname{lemma}{Lemma}{Lemmas}
\crefname{corollary}{Corollary}{Corollaries}
\crefname{remark}{Remark}{Remarks}

\newcommand{\R}{\mathbb R}
\newcommand{\T}{\mathcal T}
\newcommand{\dd}{\,\mathrm d}

\title{A quantitative Robin P\'olya--Szeg\H{o} theorem for tangential polygons}

\author{Changfeng Gui}
\address{Department  of  Mathematics,  University  of  Macau,  Macau  SAR,  P. R. China}
\address{Zhuhai UM Science and Technology Research Institute, Hengqin, Guangdong, 519031,  P. R. China}
\email{changfenggui@um.edu.mo}
\author{Yeyao Hu}
\address{School of Mathematics and Statistics, HNP-LAMA, Central South University, Changsha, Hunan 410083, P. R. China}
\email{huyeyao@gmail.com}
\author{Qinfeng Li}
\address{School of Mathematics, Hunan University, Changsha, P.R. China.}
\email{liqinfeng1989@gmail.com}
\author{Hang Yang}
\address{School of Mathematics, Hunan University, Changsha, P.R. China.}
\email{hangyang0925@gmail.com}


\hypersetup{
  pdftitle={A quantitative Robin Polya--Szego theorem for tangential polygons},
  pdfauthor={Changfeng Gui, Yeyao Hu, Qinfeng Li, Hang Yang},
  pdfkeywords={Robin torsional rigidity, quantitative Polya--Szego inequality, tangential polygon, regular polygon, shape optimization}
}

\begin{document}

\begin{abstract}
We extend the quantitative Dirichlet theorem of Gui, Hu,
and Li \cite{GuiHuLi2026} to Robin torsional rigidity. For every fixed
$\beta>0$, we prove that the regular $N$-gon uniquely
maximizes Robin torsional rigidity among tangential
$N$-gons of prescribed area. Our explicit perimeter
deficit bound contains the Dirichlet deficit obtained
by Gui, Hu, and Li \cite{GuiHuLi2026} and an additional Robin term that is
positive for nonregular tangential polygons. For triangles, we
answer the fixed-area optimization problem recorded
as open by Keady and Wiwatanapataphee \cite[Appendix A.1.2]{KeadyWiwatanapataphee2016}. We also establish
strict monotonicity along equal-area regular polygons
as the number of sides increases, with an explicit
bound for each consecutive gap. The proof combines
mixed Robin--Neumann triangular cells, spectral estimates
for angular concavity, and a quantitative scaling estimate.
\end{abstract}

\maketitle
\enlargethispage{2pt}


\section{Introduction}

Let $\Omega\subset\R^2$ be a bounded domain with Lipschitz boundary, and
let $\beta>0$. The Robin torsion function $u_{\Omega,\beta}$ is the
unique weak solution of
\begin{equation}\label{eq:Robin-problem}
\begin{cases}
-\Delta u=1 & \text{in }\Omega,\\
\partial_\nu u+\beta u=0 & \text{on }\partial\Omega,
\end{cases}
\end{equation}
where $\nu$ denotes the outward unit normal. Its integral
\(\T_\beta(\Omega):=\int_\Omega u_{\Omega,\beta}\dd x\)
is called the Robin torsional rigidity. Equivalently,
\begin{equation}\label{eq:variational-global}
\T_\beta(\Omega)=\max_{v\in H^1(\Omega)}
\left\{
2\int_\Omega v\dd x-\int_\Omega|\nabla v|^2\dd x
-\beta\int_{\partial\Omega}v^2\dd s
\right\}.
\end{equation}

After normalization, Problem~\eqref{eq:Robin-problem} models
steady heat conduction in a body with a uniform unit heat
source, with the ambient temperature taken as the reference
zero. The Robin boundary condition models convective heat
exchange with the surroundings, with the heat flux proportional
to the temperature difference across the boundary; see
\cite[Section~1]{MarusicPalokaPazanin2024}. The same problem arises
in steady unidirectional flow through a straight channel with Navier
slip at the wall. After normalization of the pressure gradient and
viscosity, $\T_\beta$ is the volume flux, and the slip length is
$\ell=\beta^{-1}$; see \cite{KeadyWiwatanapataphee2016}.

Constant-Laplacian problems with homogeneous Robin conditions were
studied systematically by Keady and McNabb
\cite{KeadyMcNabb1993}. Analytic and geometric properties of the Robin torsion function,
such as qualitative behavior, concavity, and critical-point
structure, have been studied in the literature; see, for example,
\cite{vandenBergBucur2014,Sannipoli2022,Sakaguchi1990,
CrastaFragala2021,DeRegibusGrossi2024,LiWeiYao2025}.
For a broader overview of Robin problems and related shape
optimization questions, we refer to the comprehensive survey
chapter by Bucur, Freitas, and Kennedy~\cite{BucurFreitasKennedy2017}. Recent geometric estimates for Robin torsion and the first Robin
eigenvalue were obtained in
\cite{BarbatoMasielloSannipoli2026}, and inequalities relating
these two quantities were studied by Buttazzo, Cito, and Solombrino
\cite{ButtazzoCitoSolombrino2026}. 

Shape optimization for Robin torsional rigidity has received
considerable attention. For every fixed $\beta>0$ and every bounded open Lipschitz set $\Omega\subset \mathbb{R}^{n}$, the Robin
Saint--Venant inequality states that
\begin{align}
\label{eq:SV}
\T_\beta(\Omega)\le\T_\beta(\Omega^\sharp),
\end{align}
with equality if and only if $\Omega$ is a ball, where $\Omega^\sharp$ is a ball with $|\Omega^\sharp|=|\Omega|$.
Bucur and Giacomini established this inequality through a
free-discontinuity approach in a setting that also includes domains
beyond the Lipschitz class
\cite{BucurGiacominiFreeDiscontinuity2015}; see also
\cite{BucurGiacomini2015}.
Alvino, Nitsch, and Trombetti~\cite{AAT23} established
Talenti-type comparisons using rearrangement methods,
yielding another proof of \eqref{eq:SV} in the planar case.
Masiello and Paoli \cite{MasielloPaoli2023} subsequently proved
rigidity for the planar pointwise Talenti comparison: equality can
occur only for disks.
Quantitative stability of \eqref{eq:SV} was established
by Bucur, Giacomini, and Nahon
\cite{BucurGiacominiNahon2021} using a free-discontinuity
approach, with an alternative proof via quantitative
Talenti-type comparisons given by
Amato, Barbato, Cito, Masiello, and Paoli~\cite{AmatoBarbatoCitoMasielloPaoli2026}.

An additional constraint on the number of sides leads to a
different geometric problem: for prescribed area and a fixed
$\beta>0$, does the regular $N$-gon maximize $\T_\beta$ among
$N$-gons? Its Dirichlet counterpart belongs to the classical
P\'olya--Szeg\H{o} problems for polygons
\cite{Polya1948,PolyaSzego1951}; for recent work on this problem
and related polygonal shape optimization questions, see
\cite{BucurFragala2021,KeadyDirichlet2021,KeadyRegular2022,
BogoselBucur2024,DiazAvalosLaurain2026,GuiHuLi2026} and references therein. The disk comparison \eqref{eq:SV} does not settle
this question, since the disk lies outside the admissible class,
and the usual rearrangements need not preserve the number of sides.
Moreover, classical rearrangement arguments for the Dirichlet
problem do not carry over directly to the Robin setting, where
the boundary trace is generally nonconstant and the boundary
term in the energy requires additional treatment; see
\cite{AAT23}.
A quantitative version further asks for a lower bound on the
torsional deficit in terms of geometric quantities that measure
the departure from the regular polygon. We consider this problem in the class of tangential polygons.

A convex polygon is called \emph{tangential} if it admits an
incircle tangent to all of its sides.
Every nondegenerate triangle is tangential, and for $N\ge4$
the tangential class includes some nonsymmetric convex $N$-gons. 
Related studies of torsional rigidity include explicit
calculations for the Robin problem on the equilateral triangle
\cite{Lekner2009} and work by Keady on Dirichlet torsion
and Robin torsion in tangential and regular polygons
\cite{KeadyDirichlet2021,KeadyRegular2022,KeadySlip2021}.

The direct starting point of the present work is the quantitative
Dirichlet theorem of Gui, Hu, and Li~\cite{GuiHuLi2026}. They prove
that the regular $N$-gon uniquely maximizes Dirichlet torsional
rigidity among tangential $N$-gons of prescribed area, with an
explicit lower bound for the deficit in terms of the perimeter
ratio. Their full estimate also contains a nonnegative angular
deficit. The same method yields an explicit positive gap between
the torsional rigidities of equal-area regular polygons with
consecutive numbers of sides. We now establish a Robin counterpart
for every $\beta>0$. Our estimate contains the same geometric
term as their Dirichlet perimeter bound, together with an
additional positive term depending on $\beta$.

For a planar Lipschitz domain $\Omega$, we write
$L(\Omega):=\mathcal H^1(\partial\Omega)$ for its perimeter. Our main result is the following. For $N\ge3$ and $A>0$, let
$R_{N,A}$ denote the regular $N$-gon of area $A$.

\begin{theorem}
\label{thm:perimeter-deficit}
Let $N\ge3$ and $\beta>0$. Let $P$ be a nondegenerate
tangential $N$-gon of area $A$, and let $R_{N,A}$ be
the regular $N$-gon of area $A$. Let $\mathcal{T}_{\beta}(\cdot)$ be the Robin torsional rigidity.
Writing $L(\cdot)$ for perimeter, we have
{\small
\begin{equation}\label{eq:perimeter-deficit}
\begin{aligned}
\T_\beta(R_{N,A})-\T_\beta(P)
\ge{}&
\frac{A^2}{8N\tan(\pi/N)}
\left(1-\frac{L(R_{N,A})^2}{L(P)^2}\right)^2\\
&+\frac{A^2}{2\beta L(R_{N,A})}
\left(1-\frac{L(R_{N,A})}{L(P)}\right)^2
\left(2+\frac{L(R_{N,A})}{L(P)}\right).
\end{aligned}
\end{equation}}
Moreover, $\mathcal{T}_{\beta}(R_{N,A})=\mathcal{T}_{\beta}(P)$ if and only if $P$ is regular.
\end{theorem}

For fixed $A$, $N$, and $\beta$, the lower bound depends only on
the perimeter ratio. Both terms are nonnegative, and the second
is strictly positive for every nonregular tangential polygon.
The first term agrees with the Dirichlet perimeter deficit in
\cite{GuiHuLi2026}. As $\beta\to\infty$, the second term vanishes
and Robin torsional rigidity converges to Dirichlet torsional
rigidity, recovering that estimate; see \cref{sec:large-small discussion}. 
Thus \eqref{eq:perimeter-deficit} contains an explicit additional
nonnegative contribution for every finite positive Robin parameter,
which is strictly positive whenever $P$ is nonregular. Immediately, we have the following corollaries.

\begin{corollary}\label{cor:maximizer}
Let $\mathcal{T}_{\beta}(\cdot)$ be the Robin torsional rigidity. For every $N\ge3$, $A>0$, and $\beta>0$, the regular $N$-gon of area $A$
uniquely maximizes $\T_\beta$ among tangential $N$-gons of area
$A$, up to rigid motions.
\end{corollary}

Since every nondegenerate triangle has an incircle,
\cref{thm:perimeter-deficit} also applies to the full class of triangles.

\begin{corollary}\label{cor:triangle}
Let $\mathcal{T}_{\beta}(\cdot)$ be the Robin torsional rigidity. Let $T$ be a nondegenerate triangle of area $A$, and let $R_{3, A}$
be the equilateral triangle of area $A$. For every $\beta>0$,
\[
\T_\beta(T)\le\T_\beta(R_{3, A}),
\]
with equality if and only if $T$ is equilateral. An explicit
deficit estimate is obtained from \eqref{eq:perimeter-deficit}
by setting $N=3$.
\end{corollary}

The fixed-area Robin torsion problem for triangles was explicitly
recorded as open in
\cite[Appendix~A.1.2]{KeadyWiwatanapataphee2016}. The case $N=3$ of our theorem gives not
only the maximizer but an explicit stability estimate.
For $N=4$, \cref{thm:perimeter-deficit} applies to tangential quadrilaterals and does
not assert the corresponding result for arbitrary quadrilaterals.

For the Dirichlet problem, Gui, Hu, Li, and
Zhang~\cite{GuiHuLiZhang2026} proved that the torsional rigidity
of equal-area regular polygons increases strictly with the
number of sides, using explicit estimates for a series obtained
from Schwarz--Christoffel maps and Bergman analytic content. Shortly thereafter, Gui, Hu, and Li \cite{GuiHuLi2026} provided a rather brief alternative proof of the strict monotonicity of torsional rigidity with respect to $N$.
Based on the approach in \cite{GuiHuLi2026}, our angular and scaling estimates yield the following quantitative comparison in the Robin case.
\begin{theorem}
\label{thm:N-monotonicity}
Fix $A,\beta>0$. For each integer $N\ge3$, let $R_{N,A}$
be the regular $N$-gon of area $A$. Let $\mathcal{T}_{\beta}(\cdot)$ be the Robin torsional rigidity. Writing $L(\cdot)$ for perimeter, we have
{\small
\begin{equation}\label{eq:N-monotonicity-quantitative}
\begin{aligned}
\T_\beta(R_{N+1,A})-\T_\beta(R_{N,A})&\quad\ge
\frac{A^2}{8(N+1)\tan(\pi/(N+1))}
\left(1-\frac{L(R_{N+1, A})^2}{L(R_{N, A})^2}\right)^2\\
&\qquad+
\frac{A^2}{2\beta L(R_{N+1, A})}
\left(1-\frac{L(R_{N+1, A})}{L(R_{N, A})}\right)^2
\left(2+\frac{L(R_{N+1, A})}{L(R_{N, A})}\right).
\end{aligned}
\end{equation}}
\end{theorem}

The first term again agrees with the quantitative Dirichlet
gap in \cite{GuiHuLi2026}, while the second gives a positive
Robin contribution. In particular, for fixed area and a fixed
positive Robin parameter, torsional rigidity increases strictly
with the number of sides of the regular polygon.

We now outline the proof. Let $P$ be a nondegenerate tangential
$N$-gon with inradius $r$, interior angles $\pi-2\alpha_i$,
and $a_i=\tan\alpha_i$. Then
\begin{equation*}
|P|=r^2\sum_{i=1}^N a_i,
\qquad
\sum_{i=1}^N\alpha_i=\pi.
\end{equation*}

Following the mixed-cell method of
Gui, Hu, and Li~\cite{GuiHuLi2026}, we decompose $P$ into
$2N$ right triangles by joining the incenter to the vertices
and tangency points. For each $i$, two cells are congruent
to $rD_{a_i}$, where
\[
D_a=\{(x,y):0<x<1,\ 0<y<ax\}.
\]
Let $h_b(a):=\T_b^{\mathrm{Mix}}(D_a)$ denote the cell energy
defined in \eqref{eq:def-h}, with Robin parameter $b$ on
the vertical side and homogeneous Neumann conditions on
the other two sides. Relaxing the matching conditions across
the internal cuts and using the scaling law
gives
\[
\T_\beta(P)\le 2r^4\sum_{i=1}^N h_{\beta r}(a_i),
\]
with equality for regular
polygons by reflection symmetry, see \cref{prop:relaxation}.

The comparison proceeds in two steps. First, angular concavity
compares the corrected cell energies at fixed $b=\beta r$
with those at equal angles. We then adjust the inradius
to recover the prescribed area. In the Robin setting,
this changes both the factor $r^4$ and the normalized
parameter $b=\beta r$, even though $\beta$ remains fixed. A scaling
estimate controls these changes and completes the comparison
with the regular polygon of the same area.

The two estimates are linked by the same coefficient
\[
c_b:=\frac18+\frac{3}{8b}.
\]
First, for every $b>0$, we prove that
\[
\alpha\longmapsto h_b(\tan\alpha)-c_b\tan\alpha
\quad\text{is strictly concave on }(0,\pi/2).
\]
Together with $\sum_i\alpha_i=\pi$, this allows us to apply
Jensen's inequality to compare the corrected cell energies
with those at equal angles. Second, for every $\alpha\in(0,\pi/2)$ and $b>0$,
we prove the differential scaling estimate
\[
2h_b(\tan\alpha)
+\frac b2\partial_b h_b(\tan\alpha)
\ge c_b\tan\alpha.
\]
This estimate controls the variation of the scaled cell energy
and, together with angular concavity, yields the quantitative
comparison with the regular polygon of the same area.

The paper is organized as follows. In \cref{sec:geometry}, we
introduce the angular coordinates and prove the mixed-cell
relaxation. In \cref{sec:cell-representation}, we derive the
spectral representation and study the dependence of the cell
energy on the Robin parameter. \cref{sec:quantitative comparison} establishes the angular concavity and scaling
deficit estimates and proves
\cref{thm:perimeter-deficit,cor:maximizer,cor:triangle}. In \cref{sec:N-monotonicity-short}, we prove
\cref{thm:N-monotonicity}. Finally, \cref{sec:large-small discussion}
discusses the Dirichlet and small-$\beta$ limits.

\section{Tangential polygons and mixed Robin cells}\label{sec:geometry}

Let $P$ be a nondegenerate tangential $N$-gon with incenter
$O$ and inradius $r$. Write its $i$-th interior angle as
$\pi-2\alpha_i$. Then
\begin{equation*}
0<\alpha_i<\frac{\pi}{2}
\quad\text{and}\quad
\sum_{i=1}^N\alpha_i=\pi,
\end{equation*}
where the second identity follows from the angle-sum formula for convex $N$-gons. Joining $O$ to the vertices and the tangency points decomposes $P$
into $2N$ right triangles. The two triangles adjacent to the $i$-th
vertex have legs $r$ and $r\tan\alpha_i$. Accordingly, set
\begin{equation*}
a_i:=\tan\alpha_i,
\qquad
S:=\sum_{i=1}^N a_i.
\end{equation*}
Summing the areas and boundary edges of these triangles gives
\begin{equation*}
|P|=r^2S,
\qquad
L(P)=2rS.
\end{equation*}

For later use, set
\[
a_*:=\tan\frac\pi N,
\qquad
S_*:=Na_*.
\]
The following elementary identities connect the angular variable to
the perimeter deficit.

\begin{lemma}\label{lem:geometry-ratio}
Let $P$ be a nondegenerate tangential $N$-gon of area $A$, with inradius $r$ and perimeter $L(P)$. Let
$R_{N,A}$ be the regular $N$-gon of area $A$, with inradius
$r_{N,A}$. Set
\(
z:=S/ S_*
\).
Then
\begin{equation}\label{eq:geometry-ratio}
z\ge1,
\qquad
r_{N,A}=r\sqrt z
\qquad\text{and}\quad
z=\frac{L(P)^2}{L(R_{N,A})^2}.
\end{equation}
Moreover, $z=1$ if and only if $P$ is regular.
\end{lemma}

\begin{proof}
Since the function $\tan\alpha$ is strictly convex on $(0,\pi/2)$ and $\sum_{i}\alpha_{i}=\pi$, Jensen's inequality gives
\[
S=\sum_{i=1}^N\tan\alpha_i
\ge N\tan\frac\pi N=S_*,
\]
with equality if and only if all $\alpha_i$ equal $\pi/N$. Thus
$z\ge1$. From $A=r^2S=r_{N,A}^2S_*$ we obtain
$r_{N,A}=r\sqrt z$. Finally,
\[
\frac{L(P)}{L(R_{N,A})}
=\frac{2rS}{2r_{N,A}S_*}=\sqrt z,
\]
which proves \eqref{eq:geometry-ratio}. If the angles are all equal,
every tangent segment has length $r\tan(\pi/N)$, so all sides are
equal and $P$ is regular. The converse is immediate.
\end{proof}

We now introduce the normalized right-triangle cell
\begin{equation}
 \begin{aligned}
 \label{def:Da}
 D_a&:=\{(x,y)\in\R^2:0<x<1,\ 0<y<ax\},\\
 \Gamma_a&:=\{1\}\times(0,a),
 \end{aligned}
 \qquad a>0.
\end{equation}
Thus, the two cells adjacent to the $i$-th vertex are congruent to $rD_{a_i}$, and the side corresponding to $\partial P$ is the dilated copy of $\Gamma_{a_i}$. For \(a, b, r>0\), define the mixed Robin cell energy on
\(rD_a\) by
\begin{equation}\label{eq:def-h}
 \mathcal T_b^{\mathrm{Mix}}(rD_a)
 :=
 \max_{v\in H^1(rD_a)}
 \left\{
  2\int_{rD_a}v\,\dd x
  -
  \int_{rD_a}|\nabla v|^2\,\dd x
  -
  b\int_{r\Gamma_a}v^2\,\dd s
 \right\}.
\end{equation}
Its unique maximizer is the weak solution of
\begin{equation}\label{eq:mixed-cell-pde}
 \begin{cases}
  -\Delta w=1
  & \text{in }rD_a,\\
  \partial_\nu w=0
  & \text{on }r(\partial D_a\setminus\Gamma_a),\\
  \partial_\nu w+b w=0
  & \text{on } r\Gamma_a.
 \end{cases}
\end{equation}
Since $D_a$ is a bounded Lipschitz domain, $\Gamma_a$ has positive length, and $b>0$, the standard variational theory yields a unique weak solution, equivalently, a unique maximizer of the functional. 

In the limit $b\to\infty$, the Robin condition on $\Gamma_a$ converges formally to a Dirichlet condition, while the other two sides retain homogeneous Neumann conditions. The resulting mixed Dirichlet--Neumann cell is precisely the cell used in \cite{GuiHuLi2026} for the Dirichlet torsion problem on tangential polygons. The finite-$b$ setting considered here introduces an additional difficulty: the cell energy depends on the Robin parameter $b$, which itself changes under geometric scaling. Quantitative control of this $b$-dependence is therefore essential.

Under dilation by \(r>0\), the mixed torsional rigidity satisfies
\begin{equation}\label{eq:cell-scaling}
 \mathcal T_\beta^{\mathrm{Mix}}(rD_a)
 =
 r^4\mathcal T_{\beta r}^{\mathrm{Mix}}(D_a).
\end{equation}
For simplicity, we write $h_b(a):=\mathcal T_b^{\mathrm{Mix}}(D_a)$. Thus,
\[
 \mathcal T_\beta^{\mathrm{Mix}}(rD_a)
 =
 r^4h_{\beta r}(a).
\]

\begin{proposition}\label{prop:relaxation}
 Let $P$ be a nondegenerate tangential $N$-gon with inradius $r$, and let $\T_{\beta}(P)$ be the Robin torsional rigidity of $P$. Write the $i$-th interior angle of $P$ as $\pi-2\alpha_{i}$ and set $a_{i}=\tan \alpha_{i}$, where $i=1, \ldots, N$. Let $h_{b}(a):=\T_{b}^{\mathrm{Mix}}(D_{a})$ be defined by \eqref{eq:def-h}. Then
 \begin{equation}\label{eq:relaxation}
  \T_\beta(P)
  \le
  2r^4\sum_{i=1}^Nh_{\beta r}(a_i).
 \end{equation}
 If $P$ is regular, then equality holds in \eqref{eq:relaxation}.
\end{proposition}

\begin{proof}
Joining $O$ to the vertices and tangency points decomposes $P$ into
$2N$ right triangles; the two triangles adjacent to the $i$-th vertex
are congruent to $rD_{a_i}$.  In
\eqref{eq:variational-global}, allow the functions on different cells
to vary independently, without matching their traces at the internal cuts. This enlarges the admissible space and hence gives an
upper bound.  The side inherited from $\partial P$ retains the Robin
term, whereas the two internal sides carry the natural homogeneous
Neumann condition.  Thus the relaxed maximum is the sum of the $2N$
mixed-cell energies, and \eqref{eq:cell-scaling} gives
\eqref{eq:relaxation}.

If $P$ is regular, every internal cut lies on a reflection axis of
$P$.  Uniqueness makes the Robin torsion function invariant under
these reflections; its weak normal derivative therefore vanishes on
each cut.  Its restrictions to the cells solve
\eqref{eq:mixed-cell-pde}, so the relaxation is exact.
\end{proof}

Let $A>0$, and let $R_{N,A}$ denote the regular $N$-gon of area $A$. Its inradius is
\begin{equation*}
r_{N,A}=\left(\frac{A}{N\tan(\pi/N)}\right)^{1/2}.
\end{equation*}
Since Proposition~\ref{prop:relaxation} is exact for regular polygons, it follows that
\begin{equation}\label{eq:regular-value}
\T_\beta(R_{N,A})
=2N r_{N,A}^4
h_{\beta r_{N,A}}\left(\tan\frac{\pi}{N}\right).
\end{equation}

\section{Spectral reduction of the mixed-cell energy}\label{sec:cell-representation}

We adapt the spectral reduction of Gui, Hu, and Li \cite{GuiHuLi2026} to the Robin setting. As in their Dirichlet analysis, we first map the varying triangle $D_a$ to a fixed reference triangle and then derive a spectral-measure representation of the cell energy. The Robin boundary condition introduces the additional parameter $b$, so the resulting spectral measure $\mu_b$ also depends on $b$. This dependence will be important later for computing the exact moments of $\mu_b$ and deriving the angular second-derivative estimate.

To start with, we set
\begin{equation*}
D:=\{(x,t)\in\R^2:0<t<x<1\}.
\end{equation*}
For $b>0$, define the linear functional $I$ and the symmetric bilinear forms $Q_b$ and $Y$ on $H^1(D)$ by
\begin{equation*}
\begin{aligned}
I(v)&:=\int_D v,\\
Q_b(u,v)&:=\int_D u_xv_x
+b\int_0^1u(1,t)v(1,t)\dd t,\\
Y(u,v)&:=\int_D u_tv_t.
\end{aligned}
\end{equation*}
We write $Q_b(v):=Q_b(v,v)$ and $Y(v):=Y(v,v)$. For $a>0$, the bi-Lipschitz map $(x,t)\mapsto (x,at)$
sends $D$ onto $D_a$ and $\Gamma_{1}$ onto $\Gamma_a$.  Pulling back the variational problem defining $h_b(a)$ therefore gives
\begin{equation}\label{eq:h-fixed}
h_b(a)=
\max_{v\in H^1(D)}
\left\{
2aI(v)-aQ_b(v)-a^{-1}Y(v)
\right\}.
\end{equation}
Indeed, let \(u_{a,b}\) be the unique maximizer in
\eqref{eq:def-h}, and define its pullback to \(D\) by
\[
 v_{a,b}(x,t):=u_{a,b}(x,at),
 \qquad (x,t)\in D.
\]
Then $v_{a,b}$ achieves the maximum of \eqref{eq:h-fixed} and has the value
\begin{align}
\label{eq:mix-h}
h_{b}(a)=\mathcal{T}^{\mathrm{Mix}}_{b}(D_{a})=\int_{D_{a}}u_{a,b}\, \dd x \dd y.
\end{align}

For any fixed $b>0$, we first show that $Q_b^{1/2}$ defines a norm on
$H^1(D)$. Indeed, for almost every $t\in(0,1)$ and $x\in(t,1)$,
\[
v(x,t)=v(1,t)-\int_x^1 v_x(s,t)\,\dd s.
\]
By the Cauchy--Schwarz inequality and integration over $D$, we obtain
\begin{equation}\label{eq:Q-controls-L2}
\|v\|_{L^2(D)}^2
\le
2\int_0^1 |v(1,t)|^2\dd t
+
2\int_D |v_x|^2
\le
2\max\{1,b^{-1}\}Q_b(v).
\end{equation}
In particular, $Q_b(v)=0$ implies $v=0$ almost everywhere in $D$.
Thus $Q_b$ is positive definite, and $Q_b^{1/2}$ is a norm on
$H^1(D)$. We let $\mathcal H_b$ denote the completion of $H^1(D)$
with respect to this norm. Estimate \eqref{eq:Q-controls-L2} shows that
the inclusion of $H^1(D)$ into $L^2(D)$ is continuous with respect to
the $Q_b^{1/2}$-norm and therefore extends continuously to
$\mathcal H_b$. In particular,
\[
|I(v)|
\le |D|^{1/2}\|v\|_{L^2(D)}
\le |D|^{1/2}
\bigl(2\max\{1,b^{-1}\}\bigr)^{1/2}Q_b(v)^{1/2},
\qquad v\in H^1(D),
\]
so $I$ extends uniquely to a continuous linear functional on
$\mathcal H_b$.

By the Riesz representation theorem, there exists a unique
$w_b\in\mathcal H_b$ such that
\begin{equation}\label{eq:Riesz}
Q_b(w_b,\varphi)=I(\varphi),
\qquad
\forall\varphi\in\mathcal H_b.
\end{equation}
The Riesz representative $w_b$ can be computed explicitly and is given by
\begin{align}
\label{eq:wb}
w_b(x,t)
=
\frac{1-t}{b}
+\frac{(1-t)^2-(x-t)^2}{2}.
\end{align}
Indeed, the function on the right-hand side belongs to $H^1(D)$ and
satisfies
\[
(w_b)_x(x,t)=t-x,
\qquad
b w_b(1,t)=1-t.
\]
Hence, for every $\varphi\in H^1(D)$, integration by parts in the
$x$-variable gives
\begin{align*}
Q_b(w_b,\varphi)
&=
\int_0^1
\left[
\int_t^1 (t-x)\varphi_x(x,t)\,\dd x
+
(1-t)\varphi(1,t)
\right]\dd t\\
&=
\int_0^1\int_t^1\varphi(x,t)\,\dd x\,\dd t
=
I(\varphi).
\end{align*}
Since $H^1(D)$ is dense in $\mathcal H_b$ and both sides are continuous
with respect to the $Q_b^{1/2}$-norm, the identity extends to every
$\varphi\in\mathcal H_b$, proving \eqref{eq:Riesz}.

We next consider the nonnegative form $Y$. Define the form norm
\[
\|v\|_{Y,b}:=\bigl(Q_b(v)+Y(v)\bigr)^{1/2},
\qquad v\in H^1(D).
\]
By \eqref{eq:Q-controls-L2},
\[
\|v\|_{H^1(D)}^2
\le
C_b\bigl(Q_b(v)+Y(v)\bigr)
\]
for some constant $C_b>0$. Conversely, since $D$ is a bounded
Lipschitz domain, the trace theorem yields
\[
Q_b(v)+Y(v)
\le
C_b'\|v\|_{H^1(D)}^2
\]
for some constant $C_b'>0$. Thus $\|\cdot\|_{Y,b}$ is equivalent to
the $H^1(D)$-norm on $H^1(D)$. Since $H^1(D)$ is dense in
$\mathcal H_b$ by construction and is complete with respect to the
form norm, $Y$, with form domain $H^1(D)$, is a densely defined closed
nonnegative form on $\mathcal H_b$.

Let $\mathcal K$ denote the null space of $Y$:
\begin{equation*}
\mathcal K
:=
\{q\in H^1(D):Y(q)=0\}
=
\{q\in H^1(D):q=q(x)\text{ a.e. in }D\}.
\end{equation*}
Indeed, $Y(q)=0$ if and only if $q_t=0$ almost everywhere in $D$.
By the slicing property of Sobolev functions, for almost every
$x\in(0,1)$ the function $t\mapsto q(x,t)$ belongs to $H^1(0,x)$
and has weak derivative $q_t(x,\cdot)$. It is therefore constant
almost everywhere on $(0,x)$ whenever $q_t=0$, which gives the above
description of $\mathcal K$. Moreover, $\mathcal K$ is closed in $\mathcal H_b$. To see this, let
$q_n\in\mathcal K$ and suppose that
\[
q_n\longrightarrow q
\qquad\text{in }\mathcal H_b.
\]
Since $q_n-q_m\in\mathcal K$, we have
\[
Y(q_n-q_m)=0,
\]
and hence
\[
\|q_n-q_m\|_{Y,b}^2
=
Q_b(q_n-q_m).
\]
Thus $\{q_n\}$ is Cauchy with respect to the form norm. Since $Y$ is
closed, we therefore have $q\in H^{1}(D)$. Moreover,
\[
Y(q)
=
Y(q-q_n)
\longrightarrow0,
\]
because $Y(q_n)=0$. Hence $Y(q)=0$, so $q\in\mathcal K$. This proves
that $\mathcal K$ is closed in $\mathcal H_b$.

Since $\mathcal K$ is a closed subspace of the Hilbert space
$(\mathcal H_b,Q_b)$, the $Q_b$-orthogonal projection
$P_{\mathcal K}$ onto $\mathcal K$ is well defined. We set
\(
p_b:=P_{\mathcal K}w_b
\)
and
\(
w_{1,b}:=w_b-p_b\in\mathcal K^\perp.
\)
We next determine $p_b$ explicitly. We claim that
\begin{align}
\label{eq:pb}
p_b(x)
=
\frac{1-x^2}{4}+\frac{1}{2b}.
\end{align}
Indeed, the orthogonal projection is characterized by
\begin{equation}\label{eq:projection-characterization}
Q_b(p_b,q)
=
Q_b(w_b,q)
=
I(q)
\qquad
\text{for every }q\in\mathcal K,
\end{equation}
where the second equality follows from \eqref{eq:Riesz}.

To verify \eqref{eq:pb}, direct calculation shows
\[
\int_0^1 x|q'(x)|^2\dd x=\int_D |q_x|^2
<\infty.
\]
Hence, for every $\varepsilon\in(0,1)$,
\[
\int_\varepsilon^1|q'(x)|^2 \dd x
\le \frac{1}{\varepsilon}
\int_\varepsilon^1x|q'(x)|^2 \dd x<\infty,
\]
so that $q\in H^1(\varepsilon,1)$.  Integration by parts on $(\varepsilon,1)$ gives
\begin{equation}\label{eq:check_q_integral}
-\frac12\int_\varepsilon^1x^2q'(x)\dd x
+\frac12q(1)=\frac12\varepsilon^2q(\varepsilon)
+\int_\varepsilon^1xq(x)\dd x.
\end{equation}
Since
\[
q(\varepsilon)=q(1)-\int_\varepsilon^1q'(s)\dd s,
\]
the Cauchy--Schwarz inequality yields
\begin{align*}
|q(\varepsilon)|
&\leq |q(1)|+\int_{\varepsilon}^{1}|q'(s)|\dd s\\
&\leq |q(1)|
+\left(\int_{\varepsilon}^{1}s|q'(s)|^{2}\dd s\right)^{1/2}
 \left(\int_{\varepsilon}^{1}\frac{1}{s}\dd s\right)^{1/2}\\
&\leq |q(1)|
+\left(\int_{0}^{1}s|q'(s)|^{2}\dd s\right)^{1/2}
 \left(\log \frac{1}{\varepsilon}\right)^{1/2}.
\end{align*}
Thus
\(
\varepsilon^2q(\varepsilon)\to 0\) as 
\(\varepsilon\to0\).
Letting $\varepsilon\to0$ in
\eqref{eq:check_q_integral}, we obtain
\[
Q_b(p_b,q)
=\int_0^1xq(x)\dd x
=\int_D q
=I(q).
\]
Therefore $p_b$ satisfies
\eqref{eq:projection-characterization}, and hence $p_{b}$ is the projection of $w_{b}$ onto $\mathcal{K}$.

We next remove the zero-energy directions of \(Y\). Since \(I(v)=Q_b(w_b,v)\), completing the square in \eqref{eq:h-fixed} gives
\[
2aI(v)-aQ_b(v)-a^{-1}Y(v)
=
aQ_b(w_b)-aQ_b(v-w_b)-a^{-1}Y(v).
\]
For \(v\in H^1(D)\), write $v=q+z$, where $q=P_{\mathcal K}v\in\mathcal K$. Since \(w_b=p_b+w_{1,b}\), orthogonality gives
\[
Q_b(v-w_b)
=
Q_b(q-p_b)+Q_b(z-w_{1,b}),
\qquad
Y(v)=Y(z).
\]
Minimizing first with respect to \(q\), whose unique minimizing value is
\(q=p_b\), we obtain
\begin{equation}\label{eq:h-m}
h_b(a)=aI(w_b)-m_b(a),
\end{equation}
where
\begin{equation}\label{eq:m-def}
m_b(a)
:=
\inf_{z\in\mathscr V_b}
\left\{
aQ_b(z-w_{1,b})+a^{-1}Y(z)
\right\},
\qquad
\mathscr V_b:=H^1(D)\cap\mathcal K^\perp.
\end{equation}

We now give a spectral representation of this reduced problem. Set
\(\mathscr H_b:=\mathcal K^\perp\) and \(\langle u,v\rangle_{\mathscr H_b}:=Q_b(u,v)\). Since \(\mathcal K\) is closed in \(\mathcal H_b\),
\(\mathscr H_b\) is a Hilbert space. Moreover, the restriction of \(Y\)
to \(\mathscr V_b\) is a densely defined, closed, symmetric, and
nonnegative form on \(\mathscr H_b\). Indeed, density follows by
applying \(Id-P_{\mathcal K}\) to an \(H^1(D)\)-approximation in
\(\mathcal H_b\), while closedness follows from the closedness of \(Y\)
on \(\mathcal H_b\) and the closedness of \(\mathcal K^\perp\). Since $Y$ is a densely defined, closed, nonnegative form, by the representation theorem
\cite[Definition~10.4 and Theorem~10.7]{Schmudgen2012}, there exists a
unique nonnegative self-adjoint operator \(A_b\) on \(\mathscr H_b\)
such that
\begin{equation}\label{eq:operator-domain}
Y(u,v)=Q_b(A_bu,v),
\qquad
u\in\mathcal D(A_b),\quad v\in\mathscr V_b,
\end{equation}
where 
\[\mathcal{D}(A_{b}):=\{u\in \mathscr V_b: \exists \, e_{u}\in \mathcal{K}^{\perp}\, \text{such that}\, Y(u,v)=Q_{b}(e_{u}, v)\, \text{for all }v\in\mathscr V_b \} \quad\text{and}\quad A_{b}u:=e_{u}.\]
The
square-root representation
\cite[Proposition~10.5 and Eq.~(10.12)]{Schmudgen2012}
then gives
\begin{equation*}
\mathcal D(A_b^{1/2})=\mathscr V_b,
\qquad
Y(u,v)=Q_b(A_b^{1/2}u,A_b^{1/2}v).
\end{equation*}
Furthermore, \(\ker A_b=\{0\}\). Indeed, if \(A_bu=0\), then
\[
Y(u)=Q_b(A_bu,u)=0,
\]
so that \(u\in\mathcal K\); since \(u\in\mathscr H_b=\mathcal K^\perp\),
we obtain \(u=0\).

By the spectral theorem for self-adjoint operators
\cite[Definition~4.2 and Theorem~5.7]{Schmudgen2012}, there exists a
unique projection-valued spectral measure
\(E_b\) on the Borel $\sigma$-algebra $\mathcal{B}(\mathbb{R})$
such that
\begin{equation*}
 A_b=\int_{\mathbb R}\lambda\,\dd E_b(\lambda).
\end{equation*}
More precisely, each \(E_b(B)\) is a
\(Q_b\)-orthogonal projection on \(\mathscr H_b\). Since
\(A_b\ge0\) and \(\ker A_b=\{0\}\),
\[E_b((-\infty,0))=E_b(\{0\})=0,
\qquad
E_b((0,\infty))=\operatorname{Id}_{\mathscr H_b}.\]

We define the scalar spectral measure associated with \(w_{1,b}\) by
\begin{equation}
\label{eq:def-mub}
\mu_b(B)
:=
Q_b\bigl(E_b(B)w_{1,b},w_{1,b}\bigr),
\qquad
B\in\mathcal B((0,\infty)).
\end{equation}
By the spectral theorem, \(\mu_b\) is a finite positive Borel measure
and \(\mu_b((0,\infty))=Q_b(w_{1,b})\). Moreover, for every bounded Borel function
\(g:(0,\infty)\to\mathbb R\),
\begin{equation*}
Q_b\bigl(g(A_b)w_{1,b},w_{1,b}\bigr)=\int_{(0,\infty)}g(\lambda)\dd\mu_b(\lambda).
\end{equation*}
Indeed, the corresponding identity first holds for characteristic
functions by the definition of \(\mu_b\), then for simple functions by
linearity, and finally for bounded Borel functions by bounded
approximation.  This formula converts the remaining operator minimization \eqref{eq:m-def} into a
scalar minimization over the spectral variable \(\lambda\).

We can now solve the reduced minimization problem \eqref{eq:m-def} explicitly.

\begin{proposition}
\label{prop:spectral-representation}
Let $A_{b}$ be defined by \eqref{eq:operator-domain}. For every \(a,b>0\), the infimum in \eqref{eq:m-def} is attained at
\begin{equation}\label{eq:spectral-minimizer}
 z_{a,b}
 =
 a^2(a^2\operatorname{Id}+A_b)^{-1}w_{1,b}
 \in\mathcal D(A_b),
\end{equation}
and
\begin{equation}\label{eq:m-spectral}
 m_b(a)
 =
 \int_{(0,\infty)}
 \frac{a\lambda}{a^2+\lambda}\,\dd\mu_b(\lambda).
\end{equation}
\end{proposition}

\begin{proof}
Fix \(a,b>0\), and consider the functional
\[
\mathcal F_{a,b}(z)
:=
aQ_b(z-w_{1,b})+a^{-1}Y(z),
\qquad
z\in\mathscr V_b.
\]
We first derive the variational equation that a minimizer of
\eqref{eq:m-def} must satisfy.

Suppose that \(z\in\mathscr V_b\) is a minimizer. For every
\(\eta\in\mathscr V_b\) and \(s\in\mathbb R\),
\[
\mathcal F_{a,b}(z)
\le
\mathcal F_{a,b}(z+s\eta).
\]
Expanding the right-hand side gives
\[
\begin{aligned}
\mathcal F_{a,b}(z+s\eta)
&=
\mathcal F_{a,b}(z)
+
2s\left[
aQ_b(z-w_{1,b},\eta)
+
a^{-1}Y(z,\eta)
\right]+
s^2\left[
aQ_b(\eta)+a^{-1}Y(\eta)
\right].
\end{aligned}
\]
Since the function
\(s\mapsto\mathcal F_{a,b}(z+s\eta)\)
has a minimum at \(s=0\), 
\[
aQ_b(z-w_{1,b},\eta)
+
a^{-1}Y(z,\eta)
=
0,
\qquad
\forall\eta\in\mathscr V_b.
\]
Thus any minimizer must satisfy the Euler--Lagrange equation
\begin{equation}\label{eq:Euler-spectral}
aQ_b(z-w_{1,b},\eta)
+
a^{-1}Y(z,\eta)
=
0,
\qquad
\forall\eta\in\mathscr V_b.
\end{equation}
Equivalently,
\begin{equation}\label{eq:variational-zab}
aQ_b(z,\eta)+a^{-1}Y(z,\eta)
=
aQ_b(w_{1,b},\eta),
\qquad
\forall\eta\in\mathscr V_b.
\end{equation}

We now solve \eqref{eq:variational-zab}. Since
the restriction of \(Y\) to
\(\mathscr V_b\) is closed, \(\mathscr V_b\), equipped with the form
norm
\[
\|u\|_{Y,b}
:=
\bigl(Q_b(u)+Y(u)\bigr)^{1/2},
\]
is a Hilbert space. Define the bilinear form
\[
\mathfrak b_a(u,v)
:=
aQ_b(u,v)+a^{-1}Y(u,v),
\qquad
u,v\in\mathscr V_b.
\]
For every \(u\in\mathscr V_b\),
\[
\mathfrak b_a(u,u)
=
aQ_b(u)+a^{-1}Y(u)
\ge
\min\{a,a^{-1}\}\|u\|_{Y,b}^2,
\]
so \(\mathfrak b_a\) is coercive. Moreover,
\[
|\mathfrak b_a(u,v)|
\le
\max\{a,a^{-1}\}
\|u\|_{Y,b}\|v\|_{Y,b},
\]
and hence \(\mathfrak b_a\) is continuous on
\(\mathscr V_b\times\mathscr V_b\).

The linear functional
\(
\eta\mapsto aQ_b(w_{1,b},\eta)
\)
is also continuous on \(\mathscr V_b\), since
\[
|aQ_b(w_{1,b},\eta)|
\le
aQ_b(w_{1,b})^{1/2}Q_b(\eta)^{1/2}
\le
aQ_b(w_{1,b})^{1/2}\|\eta\|_{Y,b}.
\]
The Lax--Milgram theorem therefore yields a unique
\(z_{a,b}\in\mathscr V_b\) satisfying
\eqref{eq:variational-zab}, or equivalently,
\eqref{eq:Euler-spectral}.

We next verify that this variational solution is indeed the unique
minimizer in \eqref{eq:m-def}. Let \(z\in\mathscr V_b\) be arbitrary
and write
\[
z=z_{a,b}+\eta,
\qquad
\eta\in\mathscr V_b.
\]
Expanding the functional gives
\[
\begin{aligned}
\mathcal F_{a,b}(z)
&=
aQ_b(z_{a,b}-w_{1,b})
+a^{-1}Y(z_{a,b})\\
&\qquad
+2aQ_b(z_{a,b}-w_{1,b},\eta)
+2a^{-1}Y(z_{a,b},\eta)\\
&\qquad
+aQ_b(\eta)+a^{-1}Y(\eta).
\end{aligned}
\]
The mixed terms vanish by \eqref{eq:Euler-spectral}. Hence
\[
\mathcal F_{a,b}(z)
=
\mathcal F_{a,b}(z_{a,b})
+
aQ_b(\eta)+a^{-1}Y(\eta).
\]
Since \(Q_b\) is positive definite and \(Y\) is nonnegative,
\[
aQ_b(\eta)+a^{-1}Y(\eta)\ge0,
\]
with equality if and only if \(\eta=0\). Therefore \(z_{a,b}\) is the
unique minimizer in \eqref{eq:m-def}.

We now identify \(z_{a,b}\) explicitly in terms of \(A_b\).
Multiplying \eqref{eq:variational-zab} by \(a\) and rearranging gives
\[
Y(z_{a,b},\eta)
=
Q_b\bigl(a^2(w_{1,b}-z_{a,b}),\eta\bigr),
\qquad
\forall\eta\in\mathscr V_b.
\]
By the definition of the operator associated with the closed form
\(Y\), it follows that
\(
z_{a,b}\in\mathcal D(A_b)
\)
and
\(
A_bz_{a,b}
=
a^2(w_{1,b}-z_{a,b}).
\)
Consequently,
\[
(a^2\operatorname{Id}+A_b)z_{a,b}
=
a^2w_{1,b}.
\]
Since \(A_b\) is nonnegative and self-adjoint,
\(
\sigma(A_b)\subset[0,\infty).
\)
Thus \(-a^2\notin\sigma(A_b)\), so
\(a^2\operatorname{Id}+A_b\) is invertible. Therefore
\begin{equation*}
z_{a,b}
=
a^2(a^2\operatorname{Id}+A_b)^{-1}w_{1,b},
\end{equation*}
which proves \eqref{eq:spectral-minimizer}.

It remains to compute the minimum. By the spectral functional calculus,
\[
z_{a,b}
=
\int_{(0,\infty)}
\frac{a^2}{a^2+\lambda}\,
\dd E_b(\lambda)w_{1,b},
\]
and hence
\[
z_{a,b}-w_{1,b}
=
-\int_{(0,\infty)}
\frac{\lambda}{a^2+\lambda}\,
\dd E_b(\lambda)w_{1,b}.
\]
By the definition of \(\mu_b\),
\[
Q_b(z_{a,b}-w_{1,b})
=
\int_{(0,\infty)}
\left(
\frac{\lambda}{a^2+\lambda}
\right)^2
\dd\mu_b(\lambda).
\]
Furthermore, since
\(\mathcal D(A_b^{1/2})=\mathscr V_b\), the second representation
theorem gives
\[
Y(z_{a,b})
=
Q_b(A_b^{1/2}z_{a,b})
=
\int_{(0,\infty)}
\lambda
\left(
\frac{a^2}{a^2+\lambda}
\right)^2
\dd\mu_b(\lambda).
\]
It follows that
\[
\begin{aligned}
m_b(a)
&=
aQ_b(z_{a,b}-w_{1,b})
+a^{-1}Y(z_{a,b})\\
&=
\int_{(0,\infty)}
\left[
a\left(
\frac{\lambda}{a^2+\lambda}
\right)^2
+
a^{-1}\lambda
\left(
\frac{a^2}{a^2+\lambda}
\right)^2
\right]
\dd\mu_b(\lambda)\\
&=
\int_{(0,\infty)}
\frac{a\lambda(\lambda+a^2)}
{(a^2+\lambda)^2}
\,\dd\mu_b(\lambda)\\
&=
\int_{(0,\infty)}
\frac{a\lambda}{a^2+\lambda}
\,\dd\mu_b(\lambda).
\end{aligned}
\]
This proves \eqref{eq:m-spectral}.
\end{proof}

The angular concavity in the next section requires only the total mass and
the inverse first moment of $\mu_b$.  Both can be computed exactly.

\begin{lemma}\label{lem:moments}
Let $\mu_b$ be the measure defined by \eqref{eq:def-mub}, and let
$h_b(a)$ be defined by \eqref{eq:h-fixed}. Then the total mass and the
inverse first moment of $\mu_b$ are given by
\begin{equation}\label{eq:moments}
 \int_{(0,\infty)}\dd\mu_b(\lambda)
 =\frac{b+4}{48b},
 \qquad
 \int_{(0,\infty)}\lambda^{-1}\,\dd\mu_b(\lambda)
 =\frac{1}{48}.
\end{equation}
Moreover, $h_b(a)$ admits the representation
\begin{equation}\label{eq:h-representation}
 h_b(a)
 =a\left(\frac{1}{12}+\frac{1}{3b}\right)
 -\int_{(0,\infty)}
 \frac{a\lambda}{a^2+\lambda}\,\dd\mu_b(\lambda).
\end{equation}
\end{lemma}

\begin{proof}
Direct integration of \eqref{eq:wb} and \eqref{eq:pb} gives
\begin{equation}\label{eq:direct-integrals}
 \begin{aligned}
 I(w_b)=Q_b(w_b)
 &=\int_0^1\left(\frac{(1-t)^2}{b}
 +\frac{(1-t)^3}{3}\right)\dd t
 =\frac{b+4}{12b},\\
 I(p_b)=Q_b(p_b)
 &=\int_0^1x\left(\frac{1-x^2}{4}+\frac{1}{2b}\right)\dd x
 =\frac{b+4}{16b}.
 \end{aligned}
\end{equation}
The equalities $Q_b(w_b)=I(w_b)$ and $Q_b(p_b)=I(p_b)$ follow from
\eqref{eq:Riesz} and \eqref{eq:projection-characterization}, respectively.
Since $w_b=p_b+w_{1,b}$ is an orthogonal decomposition,
\[
 \int_{(0,\infty)}\dd\mu_b
 =Q_b(w_{1,b})
 =Q_b(w_b)-Q_b(p_b)
 =\frac{b+4}{48b}.
\]

We next compute the inverse first moment of \(\mu_b\). We begin with the following basic spectral identity underlying the inverse moment:  if $w\in\mathscr H_b$ has scalar spectral measure
$\mu_w(B)=Q_b(E_b(B)w,w)$, then
\begin{equation}\label{eq:dual-spectral-characterization}
 \int_{(0,\infty)}\lambda^{-1}\,\dd\mu_w(\lambda)
 =\sup_{0\ne\varphi\in\mathscr V_b}
 \frac{|Q_b(w,\varphi)|^2}{Y(\varphi)},
\end{equation}
with the convention that both sides may be infinite.  For completeness, we include a brief proof. Set
$F_n:=E_b([n^{-1},n])$.  Here and below, \(A_b^{-1}F_n\) and \(A_b^{-1/2}F_n\) are understood in the truncated spectral sense, and hence are bounded. For $\varphi\in\mathscr V_b$, the spectral theorem
and Cauchy--Schwarz give
\[
 |Q_b(F_nw,\varphi)|
 =|Q_b(A_b^{-1/2}F_nw,A_b^{1/2}\varphi)|
 \le
 \left(\int_{[n^{-1},n]}\lambda^{-1}\,\dd\mu_w\right)^{1/2}
 Y(\varphi)^{1/2}.
\]
Since $F_nw\to w$, this proves 
\begin{align}
\label{eq:-1-leq}
\int_{(0,\infty)}\lambda^{-1}\,\dd\mu_w(\lambda)
 \geq \sup_{0\ne\varphi\in\mathscr V_b}
 \frac{|Q_b(w,\varphi)|^2}{Y(\varphi)}.
 \end{align}
Conversely, define
\(
 \varphi_n:=A_b^{-1}F_nw\in\mathcal D(A_b).
\)
Writing
$S_n:=\int_{[n^{-1},n]}\lambda^{-1}\,\dd\mu_w$, one has
\[
 Q_b(w,\varphi_n)=S_n,
 \qquad
 Y(\varphi_n)=S_n.
\]
If \(S_n>0\), then \(\varphi_n\neq0\), and hence
\[
 \frac{|Q_b(w,\varphi_n)|^2}{Y(\varphi_n)}=S_n.
\]
If \(S_n=0\), the inequality below is immediate. Thus, in either case,
\[
 \sup_{0\neq\varphi\in\mathscr V_b}
 \frac{|Q_b(w,\varphi)|^2}{Y(\varphi)}
 \geq S_n.
\]
Letting \(n\to \infty\) and applying monotone convergence, we find
\begin{align}
\label{eq:-1-geq}
\int_{(0,\infty)}\lambda^{-1}\,\dd\mu_w(\lambda)
 \leq\sup_{0\ne\varphi\in\mathscr V_b}
 \frac{|Q_b(w,\varphi)|^2}{Y(\varphi)}.
\end{align}
Together, \eqref{eq:-1-leq} and \eqref{eq:-1-geq} imply \eqref{eq:dual-spectral-characterization}.

We now apply \eqref{eq:dual-spectral-characterization} to
\(w=w_{1,b}\), whose scalar spectral measure is \(\mu_b\). From
\eqref{eq:wb} and \eqref{eq:pb},
\[(w_{1,b})_x(x,t)
=
t-\frac{x}{2},
\qquad
bw_{1,b}(1,t)
=
\frac12-t.\]
For \(\varphi\in H^1(D)\), integration by parts in the \(x\)-variable
gives
\[Q_b(w_{1,b},\varphi)
=
\frac12\int_D\varphi
-
\frac12\int_0^1t\varphi(t,t)\,\dd t.\]
On the other hand, integration by parts in the \(t\)-variable on each
interval \((0,x)\) gives
\[\int_Dt\varphi_t
=
\int_0^1x\varphi(x,x)\,\dd x
-
\int_D\varphi.\]
Combining the two identities, we
obtain
\begin{equation*}
Q_b(w_{1,b},\varphi)=
-\frac12\int_Dt\varphi_t.
\end{equation*}
Then the Cauchy--Schwarz inequality gives
\begin{equation}\label{eq:dual-CS}
|Q_b(w_{1,b},\varphi)|^2
\le
\frac14
\left(\int_Dt^2\right)
\left(\int_D|\varphi_t|^2\right)=\frac1{48}Y(\varphi),
\qquad
\forall\varphi\in H^1(D).
\end{equation}
In particular, since
\(\mathscr V_b\subset H^1(D)\),
\eqref{eq:dual-spectral-characterization} and \eqref{eq:dual-CS} imply
\begin{align}
\label{eq:dual-upper}
\int_{(0,\infty)}
\lambda^{-1}\,\dd\mu_b(\lambda)
\le
\frac1{48}.
\end{align}

It remains to show that this upper bound is attained. Let
\[\varphi_0(x,t):=-\frac{t^2}{4},
\qquad
\psi_0:=(Id-P_{\mathcal K})\varphi_0.\]
Since \(\varphi_0\in H^1(D)\) and
\(P_{\mathcal K}\varphi_0\in\mathcal K\), we have
\(\psi_0\in
H^1(D)\cap\mathcal K^\perp
=
\mathscr V_b\). Moreover, we have
\[Y(\psi_0)=Y(\varphi_0) \quad\text{and}\quad w_{1,b}\in\mathcal K^\perp.\]
Direct calculation gives
\[Y(\varphi_0)
=
\frac14\int_Dt^2
=
\frac1{48},\]
and
\[Q_b(w_{1,b},\psi_0)=Q_b(w_{1,b},\varphi_0)
=
-\frac12\int_D
t\left(-\frac t2\right)
=
\frac14\int_Dt^2
=
\frac1{48}.\]
Consequently, 
\[\frac{|Q_b(w_{1,b},\psi_0)|^2}{Y(\psi_0)}
=
\frac{(1/48)^2}{1/48}
=
\frac1{48}.\]
Combining this with \eqref{eq:dual-upper}, we obtain
\[\sup_{0\ne\varphi\in\mathscr V_b}
\frac{|Q_b(w_{1,b},\varphi)|^2}{Y(\varphi)}
=
\frac1{48}.\]

In view of \eqref{eq:dual-spectral-characterization},
\[\int_{(0,\infty)}
\lambda^{-1}\,\dd\mu_b(\lambda)
=
\frac1{48},\]
which proves the second identity in \eqref{eq:moments}. Finally, \eqref{eq:h-representation} follows
from \eqref{eq:h-m}, \eqref{eq:m-spectral} and \eqref{eq:direct-integrals}.

\end{proof}

\begin{remark}\label{rem:Robin-moments}
The inverse moment in \eqref{eq:moments} is independent of $b$, while
the total mass splits as
\[
\mu_b((0,\infty))=\frac1{48}+\frac1{12b}.
\]
Formally, the first term is the Dirichlet mass from
\cite{GuiHuLi2026}, and the second records the finite Robin boundary
energy. 
\end{remark}

\begin{proposition}\label{prop:b-differentiability}
Let $h_b(a)$ be defined by \eqref{eq:h-fixed}. The map
\(h(a,b):=h_b(a)\)
belongs to \(C^1((0,\infty)^2)\). More precisely, let \(u_{a,b}\) be the unique maximizer in
\eqref{eq:def-h}, and define its pullback to \(D\) by
\[
 v_{a,b}(x,t):=u_{a,b}(x,at),
 \qquad (x,t)\in D.
\]
Then
\begin{align}
 \partial_a h_b(a)
 &=
 2I(v_{a,b})
 -Q_b(v_{a,b})
 +a^{-2}Y(v_{a,b}),
 \label{eq:envelope-a}\\
 \partial_b h_b(a)
 &=
 -a\int_0^1v_{a,b}(1,t)^2\,\dd t
 =
 -\int_{\Gamma_a}u_{a,b}^2\,\dd s.
 \label{eq:envelope}
\end{align}
Both partial derivatives are continuous on \((0,\infty)^2\).
\end{proposition}

\begin{proof}
For $a,b>0$, set
\[
 \mathcal J_{a,b}(v)
 :=2aI(v)-aQ_b(v)-a^{-1}Y(v),
 \qquad v\in H^1(D),
\]
and define
\[
 \mathcal B_{a,b}(w,\varphi)
 :=aQ_b(w,\varphi)+a^{-1}Y(w,\varphi).
\]
The change of variables $(x,t)\mapsto(x,at)$ is bi-Lipschitz from
$D$ onto $D_a$.  Thus $v_{a,b}$ is the unique maximizer of
$\mathcal J_{a,b}$, and its Euler equation is
\begin{equation}\label{eq:euler}
 \mathcal B_{a,b}(v_{a,b},\varphi)=aI(\varphi),
 \qquad \varphi\in H^1(D).
\end{equation}

We first prove that
\begin{equation}\label{eq:continuity-vab}
 (a,b)\longmapsto v_{a,b}
 \quad\text{is continuous from }(0,\infty)^2\text{ to }H^1(D).
\end{equation}
Fix a compact rectangle
$K=[a_-,a_+]\times[b_-,b_+]\Subset(0,\infty)^2$. For any $\varphi \in H^{1}(D)$, the estimate in \eqref{eq:Q-controls-L2} gives
\[
 \|\varphi \|_{L^2(D)}^2
 \leq2\|\varphi (1,\cdot)\|_{L^2(0,1)}^2
 +2\|\varphi_x\|_{L^2(D)}^2.
\]
Therefore, with
\[
 c_0:=\min\{a_-,a_+^{-1},a_-b_-\}>0,
\]
one has, uniformly for $(a,b)\in K$,
\begin{align}
\label{eq:B-coercivity}
 \mathcal B_{a,b}(\varphi, \varphi)
 \geq c_0\left(
 \|\varphi_x\|_{L^2(D)}^2+
 \|\varphi_t\|_{L^2(D)}^2+
 \|\varphi(1,\cdot)\|_{L^2(0,1)}^2
 \right)\geq\frac{c_0}{3}\|\varphi\|_{H^1(D)}^2.
\end{align}
Taking \(\varphi=v_{a,b}\) in \eqref{eq:euler} and using
\eqref{eq:B-coercivity}, we obtain
\[\frac{c_0}{3}\|v_{a,b}\|_{H^1(D)}^2
\le
a|I(v_{a,b})|
\le
a_+|D|^{1/2}\|v_{a,b}\|_{H^1(D)}.\]
Hence 
\[\sup_{(a,b)\in K}\|v_{a,b}\|_{H^1(D)}<\infty.\]

Let $(a_n,b_n)\to(a,b)$ in $K$, and abbreviate
$v_n=v_{a_n,b_n}$ and $v=v_{a,b}$.  Subtracting their Euler equations
gives
\begin{align}
\label{eq:diff-B}
 \mathcal B_{a,b}(v_n-v,\varphi)
 =(a_n-a)I(\varphi)
 +\bigl(\mathcal B_{a,b}-\mathcal B_{a_n,b_n}\bigr)(v_n,\varphi).
\end{align}
By trace continuity, there are numbers $\varepsilon_n\to0$ such that
\[
 \left|
 \bigl(\mathcal B_{a,b}-\mathcal B_{a_n,b_n}\bigr)(w,\varphi)
 \right|
 \leq\varepsilon_n
 \|w\|_{H^1(D)}\|\varphi\|_{H^1(D)}.
\]
Taking $\varphi=v_n-v$ in \eqref{eq:B-coercivity} and \eqref{eq:diff-B}, and dividing by $\|v_n-v\|_{H^1(D)}$ when this norm is nonzero, we
obtain
\[
 \frac{c_0}{3}\|v_n-v\|_{H^1(D)}
 \leq |a_n-a||D|^{1/2}
 +\varepsilon_n\|v_n\|_{H^1(D)}.
\]
The right-hand side tends to zero, proving
\eqref{eq:continuity-vab}.

For $s\neq0$ with $a+s>0$, maximality gives
\begin{align*}
 \mathcal J_{a+s,b}(v_{a,b})-
   \mathcal J_{a,b}(v_{a,b})\leq h_b(a+s)-h_b(a)\leq
 \mathcal J_{a+s,b}(v_{a+s,b})-
 \mathcal J_{a,b}(v_{a+s,b}).
\end{align*}
For any $v\in H^1(D)$,
\[
 \frac{\mathcal J_{a+s,b}(v)-\mathcal J_{a,b}(v)}{s}
 =2I(v)-Q_b(v)+\frac{Y(v)}{a(a+s)}.
\]
After division by $s$, the order of the preceding bounds reverses
when $s<0$, but both endpoints converge to the same limit by
\eqref{eq:continuity-vab}.  This proves \eqref{eq:envelope-a}.
The identical two-sided comparison with $b+s$ in place of $a+s$
gives
\[
 \partial_b h_b(a)
 =-a\int_0^1v_{a,b}(1,t)^2\dd t.
\]
The substitution $y=at$ yields the second identity in
\eqref{eq:envelope}.

Finally, \eqref{eq:continuity-vab} and trace continuity show directly
from \eqref{eq:envelope-a}--\eqref{eq:envelope} that both partial
derivatives are continuous.  Hence
$h\in C^1((0,\infty)^2)$.
\end{proof}

\section{Quantitative comparison for tangential polygons}\label{sec:quantitative comparison}

The rational kernel used below is the same as in the
Dirichlet spectral argument of Gui, Hu, and Li
\cite{GuiHuLi2026}. Their majorant remains effective for
$\xi\le1$, but does not yield the required bound for
$\xi>1$ because the Robin spectral measure has different
moments. We therefore introduce a second majorant for
the large-angle regime, adapted to the Robin moment identities.

\begin{proposition}\label{prop:angular-curvature}
Let $b>0$, and let $h_b$ be defined by \eqref{eq:h-fixed}. Then the function
\[
 g_b(\alpha)
 :=
 h_b(\tan\alpha)
 -
 \left(\frac18+\frac{3}{8b}\right)\tan\alpha,
 \qquad
 0<\alpha<\frac\pi2,
\]
is strictly concave on $(0,\pi/2)$.
\end{proposition}

\begin{proof}
Let $\alpha\in (0,\pi/2)$, and let $a=\tan \alpha$. For $\xi,\lambda>0$, introduce
\begin{equation*}
 K_\xi(\lambda)
 :=
 \frac{2\lambda(3\xi\lambda-\xi-\lambda^2+3\lambda)}
 {(\xi+\lambda)^3}.
\end{equation*}
A direct differentiation gives
\begin{equation*}
 \frac{\dd^2}{\dd\alpha^2}
 \left(\frac{a\lambda}{a^2+\lambda}\right)
 =-a(1+a^2)K_{a^2}(\lambda).
\end{equation*}
Since $\mu_b$ has finite mass and the differentiated kernel is locally
uniformly bounded for $a>0$, differentiation under the integral sign in
\eqref{eq:h-representation} is justified.  Using
$a''=2a(1+a^2)$, we obtain
\begin{align*}
\frac{d^{2}}{d \alpha^{2}}g_{b}(\alpha)=a(1+a^{2})\Big(\int_{(0,\infty)} K_{a^{2}}(\lambda) \dd\mu_{b}(\lambda)-\frac{1}{12}-\frac{1}{12b}\Big).
\end{align*}
Thus it remains to prove for every $\xi>0$, 
\begin{align}
\label{eq:K-upper}
\int_{(0,\infty)} K_{\xi}(\lambda) \dd\mu_{b}(\lambda)<\frac{1}{12}+\frac{1}{12b}.
\end{align}

We prove this estimate in two ranges.
\medskip

Assume first that \(0<\xi\le1\). We claim that for every $\xi \in (0,1]$ and all $\lambda>0$,
\begin{equation}\label{eq:K-bound-small}
K_\xi(\lambda)
<
4+\frac{6}{1+\xi}
\left(\frac1\lambda-1\right).
\end{equation}

Indeed, after multiplying the difference, the right-hand side minus the left-hand side, by the positive factor $\frac{1}{2}\xi^{-1}\lambda(\xi+\lambda)^{3}(1+\xi)$, we obtain
\[\begin{aligned}
P_{1}(\xi, \lambda)
=
2\lambda\xi^3
+(6\lambda^2-\lambda+3)\xi^2+(3\lambda^3-2\lambda^2+9\lambda)\xi
+\lambda^2(3\lambda^2-9\lambda+10).
\end{aligned}
\]
All coefficients of this polynomial in \(\xi\) are strictly positive
for \(\lambda>0\). Indeed,

\[6\lambda^2-\lambda+3>0,\quad 3\lambda^2-2\lambda+9>0\quad \text{and}\quad 3\lambda^2-9\lambda+10>0,\]
because the three quadratic polynomials have negative
discriminants. Hence \(P_{1}(\xi, \lambda)>0\), proving
\eqref{eq:K-bound-small}.
Write
\[A_\xi:=\frac{4\xi-2}{1+\xi},
\qquad
B_\xi:=\frac6{1+\xi}.\]
Then \eqref{eq:K-bound-small} becomes
\[K_\xi(\lambda)
<
A_\xi+\frac{B_\xi}{\lambda}.\]
Since
\[
A_\xi+B_\xi=4,
\qquad
A_\xi\le1,
\]
the two moment identities in \eqref{eq:moments} give

\[
\begin{aligned}
\int_{(0,\infty)} K_\xi\,\dd\mu_b
<
A_\xi\frac{b+4}{48b}
+
B_\xi\frac1{48}=
\frac{A_\xi+B_\xi}{48}
+\frac{A_\xi}{12b}\le
\frac1{12}+\frac1{12b}.
\end{aligned}
\]
The first inequality is strict because
\eqref{eq:K-bound-small} is strict for every \(\lambda>0\) and
\(\mu_b((0,\infty))
=
\frac{b+4}{48b}>0\).
Thus \eqref{eq:K-upper} holds for \(0<\xi\le1\).
\medskip

We next consider \(\xi\ge1\). We claim that for every $\xi \geq 1$ and all $\lambda>0$
\begin{equation}\label{eq:K-bound-large}
K_\xi(\lambda)
<
\frac12+\frac{7}{2\lambda}.
\end{equation}

After multiplying the difference of \eqref{eq:K-bound-large}, the
right-hand side minus left-hand side, by the positive factor
$2\lambda(\xi+\lambda)^3$, we obtain
\begin{equation*}
 \begin{aligned}
 P_{2}(\xi,\lambda)
 =5\lambda^4-(9\xi+5)\lambda^3
 +(3\xi^2+25\xi)\lambda^2+(\xi^3+21\xi^2)\lambda+7\xi^3.
 \end{aligned}
\end{equation*}
Writing $\lambda=t\xi$ gives
\begin{equation*}
 \frac{P_{2}(\xi,t\xi)}{\xi^3}
 =\xi t(t-1)^2(5t+1)-5t^3+25t^2+21t+7.
\end{equation*}
Since $\xi\ge1$, the right-hand side is bounded below by
\begin{equation*}
 q(t):=5t^4-14t^3+28t^2+22t+7.
\end{equation*}
For $0<t\le1$,
\[
 q(t)=5t^4+14t^2(2-t)+22t+7>0,
\]
whereas, for $t=1+s\ge1$,
\[
 q(1+s)=5s^4+6s^3+16s^2+56s+48>0.
\]
Thus $P_{2}(\xi,\lambda)>0$, and \eqref{eq:K-bound-large} follows.
Integrating \eqref{eq:K-bound-large} and using
\eqref{eq:moments}, we get
\begin{align*}
\int_{(0,\infty)} K_\xi\,\dd\mu_b
<
\frac12\frac{b+4}{48b}
+
\frac72\frac1{48}
=
\frac1{12}+\frac1{24b}<
\frac1{12}+\frac1{12b}.
\end{align*}
Therefore \eqref{eq:K-upper} also holds for \(\xi\ge1\).
Hence \eqref{eq:K-upper} holds for every \(\xi>0\), which proves the strict concavity of $g_{b}$.
\end{proof}

\begin{proposition}\label{pro:scaling-deficit}
Let $a, b>0$, and let $h_b$ be defined by \eqref{eq:h-fixed}. Then for every $z\ge1$, we have
\begin{equation}\label{eq:scaling-deficit}
\begin{aligned}
z^2h_{b\sqrt z}(a)-h_b(a)-a(\frac{1}{8}+\frac{3}{8b})(z-1)
\ge{}&\frac a{16}(z-1)^2
+\frac a{8b}\rho(z),
\end{aligned}
\end{equation}
where $\rho(z)=2z^{3/2}-3z+1$. 
\end{proposition}

\begin{proof}
For $s\ge1$, set
\[
F(s):=s^2h_{b\sqrt s}(a).
\]
The $C^1$ regularity from \cref{prop:b-differentiability} and the
chain rule give
\begin{align}
\label{eq:dF}
F'(s)
&=s\left[
2h_{b\sqrt s}(a)
+\frac{b\sqrt s}{2}\partial_b h_{b\sqrt s}(a)
\right].
\end{align}

We now claim that for every $a,b>0$, 
\begin{align}
\label{eq:scaling-estimate}
2h_{b}(a)+\frac{b}{2}\partial_{b}h_{b}(a)\geq a(\frac{1}{8}+\frac{3}{8b}).
\end{align}
Let $u=u_{a,b}$ be the maximizer in \eqref{eq:def-h}, and set
\[
G:=\int_{D_a}|\nabla u|^2,
\qquad
Z:=\int_{\Gamma_a}u^2.
\]
From \eqref{eq:mixed-cell-pde} and \eqref{eq:mix-h}, we have
\begin{equation}\label{eq:h-GZ}
h_b(a)=\int_{D_a}u=G+bZ.
\end{equation}
Combining \eqref{eq:h-GZ} with \eqref{eq:envelope} 
yields
\begin{equation}\label{eq:combined-GZ}
2h_b(a)+\frac b2\partial_bh_b(a)
=2G+\frac{3b}{2}Z.
\end{equation}

Let $\phi(x,y)=(1-x^2)/4$. Since $\phi=0$ on $\Gamma_a$, the weak
equation gives
\[
\int_{D_a}\nabla u\cdot\nabla\phi
=\int_{D_a}\phi=\frac a{16},
\qquad
\int_{D_a}|\nabla\phi|^2=\frac a{16}.
\]
Cauchy--Schwarz therefore gives
\begin{equation}\label{eq:G-lower}
G\ge\frac a{16}.
\end{equation}
Testing the weak equation \eqref{eq:mixed-cell-pde} with $1$ gives
\[
b\int_{\Gamma_a}u=\frac a2.
\]
Since $\mathcal H^1(\Gamma_a)=a$, another application of
Cauchy--Schwarz yields
\begin{equation}\label{eq:Z-lower}
Z\ge\frac{a}{4b^2}.
\end{equation}
Substituting \eqref{eq:G-lower} and \eqref{eq:Z-lower} into \eqref{eq:combined-GZ} proves
\eqref{eq:scaling-estimate}.

Combining \eqref{eq:dF} and \eqref{eq:scaling-estimate} with parameter $b\sqrt{s}$, we have
\begin{align*}
F'(s)\geq sa(\frac{1}{8}+\frac{3}{8b\sqrt{s}})=a(\frac{s}{8}+\frac{3\sqrt{s}}{8b}).
\end{align*}
Integrating from $1$ to $z$ gives
\[
F(z)-F(1)
\ge a\left[
\frac{z^2-1}{16}
+\frac{z^{3/2}-1}{4b}
\right].
\]
Subtracting
$a(\frac{1}{8}+\frac{3}{8b})(z-1)=a[(z-1)/8+3(z-1)/(8b)]$ gives exactly
\eqref{eq:scaling-deficit}.  
\end{proof}

\begin{remark}
The common coefficient $c_b=\frac18+\frac{3}{8b}$ in
\cref{prop:angular-curvature,pro:scaling-deficit}
ensures that the linear terms in the angular and scaling
estimates cancel, leaving the Dirichlet-type quadratic
deficit and the additional Robin remainder
$\rho(z)/(8b)$, which is positive for $z>1$.
\end{remark}

\begin{proof}[Proof of \cref{thm:perimeter-deficit}]
Let  $N\geq 3$ be an integer, and let $P$ be a nondegenerate tangential $N$-gon of area $A$, inradius $r$, and perimeter $L(P)$. Let $R_{N,A}$ be the regular $N$-gon of area $A$. Write the $i$-th interior angle of $P$ as $\pi- 2\alpha_{i}$,
\[a_{i}=\tan \alpha_{i}, \quad a_{*}=\tan \frac{\pi}{N}, \quad b=\beta r, \quad z=\frac{L(P)^{2}}{L( R_{N,A})^{2}}.\]

Let $h_b$ be defined by \eqref{eq:h-fixed}, and let $c_{b}=\frac{1}{8}+\frac{3}{8b}$.

By
\cref{lem:geometry-ratio}, the regular polygon $R_{N,A}$ has inradius
$r\sqrt z$ and effective Robin parameter $b\sqrt z$. The mixed-cell
relaxation \eqref{eq:relaxation} and its exactness for the regular
polygon give
\begin{equation}\label{eq:main-quantitative}
\begin{aligned}
\T_\beta(R_{N,A})-\T_\beta(P)
&\ge 2r^4\left[
Nz^2h_{b\sqrt z}(a_*)-\sum_{i=1}^Nh_b(a_i)
\right]\\
&=2r^{4} \Big[N\big(h_{b}(a_{*})-c_{b}a_{*}\big)-\sum_{i=1}^{N}\big(h_{b}(a_{i})-c_{b}a_{i} \big)\Big]\\
&+2Nr^{4} \Big[z^{2}h_{b\sqrt{z}}(a_{*})-h_{b}(a_{*})-a_{*}c_{b}(z-1) \Big].
\end{aligned}
\end{equation}

By \cref{prop:angular-curvature}, the function
\[\alpha \mapsto h_{b}(\tan \alpha)-c_{b}\tan \alpha\]
is strictly concave. Since
$\sum_i\alpha_i=\pi$, Jensen's inequality applies and yields
\begin{equation}\label{eq:main-Jensen}
N\big(h_{b}(a_{*})-c_{b}a_{*}\big)-\sum_{i=1}^{N}\big(h_{b}(a_{i})-c_{b}a_{i} \big)\geq 0
\end{equation}
with equality if and only if all $\alpha_i=\pi/N$, $i=1, 2, \ldots, N$. Applying
\cref{pro:scaling-deficit} gives
\begin{align}
\label{eq:main-scaling}
z^{2}h_{b\sqrt{z}}(a_{*})-h_{b}(a_{*})-a_{*}c_{b}(z-1)\geq \frac{a_{*}}{16}(z-1)^{2}+\frac{a_{*}}{8b}\rho(z)
\end{align}
whose right-hand side vanishes if and only if $z=1$. From \cref{lem:geometry-ratio}, $z=1$ if and only if $P$ is a regular $N$-gon. 

Drop the nonnegative angular deficit \eqref{eq:main-Jensen} in \eqref{eq:main-quantitative}. Since $S_*=Na_*$,
$A=r^2zS_*$, and $z=L(P)^{2}/L(R_{N, A})^{2}$, direct algebra gives
\[
\frac{r^4S_*}{8}(z-1)^2
=
\frac{A^2}{8S_*}
\left(1-\frac1z\right)^2.
\]
For the Robin term, we have
\[
\begin{aligned}
\frac{r^4S_*}{4b}\rho(z)
&=\frac{A^2}{2\beta L(R_{N,A})}
(1-z^{-\frac{1}{2}})^2(2+z^{-\frac{1}{2}}).
\end{aligned}
\]
Substituting $S_*=N\tan(\pi/N)$ proves
\eqref{eq:perimeter-deficit}. 

If $P$ is not regular, the angles $\alpha_i$ are not all equal,
by \cref{lem:geometry-ratio}. Since
$\sum_{i=1}^N\alpha_i=\pi$ and $g_b$ is strictly concave,
the angular deficit in \eqref{eq:main-Jensen} is strictly
positive. Hence \eqref{eq:main-quantitative} and
\eqref{eq:main-scaling} give strict inequality in
\eqref{eq:perimeter-deficit}.
Conversely, if $P$ is regular, it is congruent to $R_{N,A}$,
and both sides of \eqref{eq:perimeter-deficit} vanish.
This proves the equality statement.
\end{proof}

\begin{proof}[Proof of \cref{cor:maximizer}]
The assertion follows immediately from
\cref{thm:perimeter-deficit} and its equality case.
\end{proof}

\begin{proof}[Proof of \cref{cor:triangle}]
Every nondegenerate triangle is a convex tangential $3$-gon. Apply
\cref{thm:perimeter-deficit} with
$N=3$ and $\tan(\pi/3)=\sqrt3$.
\end{proof}

\section{Monotonicity along equal-area regular polygons}
\label{sec:N-monotonicity-short}
Let $c_{b}=\frac{1}{8}+\frac{3}{8b}$. For $b>0$, set
\[
g_b(\alpha):=h_b(\tan\alpha)-c_b\tan\alpha.
\]
By \cref{prop:angular-curvature}, $g_b$ is strictly concave on
$(0,\pi/2)$. We shall use the endpoint $0$. From
\eqref{eq:h-representation} and
$0\le a\lambda/(a^2+\lambda)\le a$, one has $h_b(a)=O(a)$ as
$a\downarrow0$. Hence $g_b(\alpha)\to0$ as $\alpha\downarrow0$.
Defining $g_b(0)=0$ gives a continuous concave extension to
$[0,\pi/2)$.

\begin{proof}[Proof of \cref{thm:N-monotonicity}]
Let $k\geq 3$ be an integer. Define
\[
a_k:=\tan\frac\pi k,
\qquad
S_k:=ka_k,
\qquad
r_k:=\sqrt{\frac A{S_k}}.
\]
Let
\[
b:=\beta r_{k},
\qquad
z_{k}:=\frac{S_k}{S_{k+1}}.
\]
Then $r_{k+1}=r_{k}\sqrt z_{k}$ and $\beta r_{k+1}=b\sqrt z_{k}$.

We first verify $z_{k}>1$. The function
$q(x)=\tan x/x$ satisfies
\[
q'(x)=\frac{x-\sin x\cos x}{x^2\cos^2x}>0,
\qquad 0<x<\frac\pi2,
\]
because the numerator vanishes at $0$ and has derivative
$2\sin^2x>0$. Since
\(
S_k=\pi q(\pi/k)
\), we have $S_k>S_{k+1}$. The exact regular-cell formula \eqref{eq:regular-value} gives
\begin{equation}\label{eq:N-difference}
\frac{\T_\beta(R_{k+1,A})-\T_\beta(R_{k,A})}{2r_k^4}
=(k+1)z_{k}^2h_{b\sqrt z_{k}}(a_{k+1})-kh_b(a_k).
\end{equation}
Apply Jensen's inequality to the $(k+1)$ angles
\[
0,\underbrace{\frac\pi k,\ldots,\frac\pi k}_{k\text{ times}},
\]
whose mean is $\pi/(k+1)$. Let $g_{b}$ be defined as in  \cref{prop:angular-curvature}. Concavity and $g_b(0)=0$ yield
\[
kg_b\left(\frac\pi k\right)
\le (k+1)g_b\left(\frac{\pi}{k+1}\right).
\]
Since $S_{k}-S_{k+1}=S_{k+1}(z_{k}-1)=(k+1)a_{k+1}(z_{k}-1)$, this is equivalent to
\begin{equation}\label{eq:N-Jensen}
kh_b(a_{k})
\le (k+1)\bigl[h_b(a_{k+1})+a_{k+1}c_b(z_{k}-1)\bigr].
\end{equation}
Combining \eqref{eq:N-difference}, \eqref{eq:N-Jensen}, and
\cref{pro:scaling-deficit} gives
\[
\begin{aligned}
\frac{\T_\beta(R_{k+1,A})-\T_\beta(R_{k,A})}{2r_{k}^4}
&\ge (k+1)\left[
\frac{a_{k+1}}{16}(z_{k}-1)^2+\frac{a_{k+1}}{8b}\rho(z_{k})
\right]\\
&=\frac{S_{k+1}}{16}(z_{k}-1)^2
+\frac{S_{k+1}}{8b}\rho(z_{k}).
\end{aligned}
\]
Multiplication by $2r_{k}^4$ and then setting $k=N$ proves
\eqref{eq:N-monotonicity-quantitative}. Since $z_{k}>1$,
$\rho(z_{k})>0$ and the lower bound is strictly positive.
\end{proof}

\begin{remark}
Retaining only the first term gives a $\beta$-independent
lower bound that agrees with the quantitative Dirichlet
monotonicity gap in \cite{GuiHuLi2026}. The second term is
an additional positive remainder from the Robin scaling
estimate comparing the inradii $r_k$ and $r_{k+1}$.
\end{remark}

\section{Large and small Robin parameter limits}
\label{sec:large-small discussion}

We briefly discuss the limits $\beta\to\infty$ and
$\beta\downarrow0$ and their consequences for our estimates.
For asymptotic expansions of Robin energies in these regimes,
see \cite{ButtazzoOgnibene2026} and the references therein.
Here we give direct arguments for the limits needed below.

\subsection{The limit $\beta\to\infty$}

Let $\T_D$ denote Dirichlet torsional rigidity. For every fixed
bounded Lipschitz domain $\Omega$,
\begin{equation}\label{eq:global-Dirichlet-limit}
\T_\beta(\Omega)\longrightarrow\T_D(\Omega)
\qquad\text{as }\beta\to\infty.
\end{equation}
Indeed, let $u_{\Omega,\beta}$ be the Robin torsion function, so that
\begin{align}
\label{eq:weak-robin}
\int_\Omega \nabla u_{\Omega,\beta}\cdot\nabla\varphi
+\beta\int_{\partial\Omega}u_{\Omega,\beta}\varphi
=
\int_\Omega\varphi
\qquad
\text{for every }\varphi\in H^1(\Omega).
\end{align}
By Friedrichs' inequality
\cite[Theorem~1.9]{Necas2012}, there exists a constant
$C_{\Omega}>0$ such that
\[
\|v\|^{2}_{H^{1}(\Omega)}
\le
C_\Omega
\left(
\|\nabla v\|_{L^2(\Omega)}^2
+
\|v\|_{L^2(\partial\Omega)}^2
\right).
\]
Consequently, for every $\beta_0>0$ and $\beta\ge\beta_0$,
\begin{align}
\label{eq:PT}
\|v\|_{H^1(\Omega)}^2
\le
C_{\Omega,\beta_0}
\left(
\|\nabla v\|_{L^2(\Omega)}^2
+
\beta\|v\|_{L^2(\partial\Omega)}^2
\right), \qquad	\beta\geq\beta_{0}>0.
\end{align}
Testing \eqref{eq:weak-robin} with
$\varphi=u_{\Omega,\beta}$ and applying \eqref{eq:PT} with
$v=u_{\Omega,\beta}$ gives
\[
\|u_{\Omega,\beta}\|_{H^1(\Omega)}
\le C,
\qquad \beta\ge\beta_0,
\]
and, at the same time,
\[
\beta\,
\|u_{\Omega,\beta}\|_{L^2(\partial\Omega)}^2
\le
\int_\Omega u_{\Omega,\beta}
\le C
\]
for some constant $C>0$ depending only on $\Omega$ and $\beta_{0}$.
Hence, along any sequence $\beta_j\to\infty$, after passing to a
subsequence,
\[
u_{\Omega,\beta_j}\rightharpoonup u
\qquad\text{weakly in }H^1(\Omega),
\]
while
\[
\|u_{\Omega,\beta_j}\|_{L^2(\partial\Omega)}^2
\le\frac{C}{\beta_j}\longrightarrow0.
\]
By continuity of the trace operator, $u$ has zero trace and hence
$u\in H_0^1(\Omega)$. Moreover, taking $\varphi\in H^{1}_{0}(\Omega)$ in \eqref{eq:weak-robin} eliminates the boundary term, and passing
to the limit yields
\[
\int_\Omega\nabla u\cdot\nabla\varphi
=
\int_\Omega\varphi
\qquad
\text{for every }\varphi\in H_0^1(\Omega).
\]
Thus $u$ is the Dirichlet torsion function. By uniqueness of the Dirichlet torsion function, the whole family
converges weakly to $u$ in $H^1(\Omega)$. Since
$v\mapsto\int_\Omega v$ is a continuous linear functional on
$H^1(\Omega)$, \eqref{eq:global-Dirichlet-limit} follows.

For a fixed tangential $N$-gon $P$ of area $A$, we may therefore let
$\beta\to\infty$ in \eqref{eq:perimeter-deficit} to obtain
\[
\T_D(R_{N,A})-\T_D(P)
\ge
\frac{A^2}{8N\tan(\pi/N)}
\left(
1-\frac{L(R_{N,A})^2}{L(P)^2}
\right)^2.
\]
Thus our perimeter estimate recovers the quantitative Dirichlet
bound of Gui, Hu, and Li~\cite{GuiHuLi2026}. Likewise,
\eqref{eq:N-monotonicity-quantitative} recovers their quantitative
monotonicity bound for equal-area regular polygons, since the
additional Robin term vanishes as $\beta\to\infty$.

\subsection{The limit $\beta\downarrow0$}
We next consider the regime $\beta\downarrow0$. For every fixed
bounded Lipschitz polygon $P$,
\begin{equation}\label{eq:small-beta-torsion}
\T_\beta(P)
=
\frac{|P|^2}{\beta L(P)}+O(1)
\qquad\text{as }\beta\downarrow0.
\end{equation}
We include a direct argument. Let $u_\beta:=u_{P,\beta}$. Taking $\varphi\equiv 1$ in \eqref{eq:weak-robin} gives
\[
\beta\int_{\partial P}u_\beta=|P|.
\]
Thus the boundary mean of $u_\beta$ is
\[
m_\beta
:=
\frac1{L(P)}\int_{\partial P}u_\beta
=
\frac{|P|}{\beta L(P)}.
\]
Set
\(
w_\beta:=u_\beta-m_\beta\).
Then
\[
\int_{\partial P}w_\beta=0.
\]
Testing \eqref{eq:weak-robin} with $\varphi=w_\beta$ and using
$\int_{\partial P}w_\beta=0$ gives
\[
\int_P|\nabla w_\beta|^2
+
\beta\int_{\partial P}w_\beta^2
=
\int_P w_\beta.
\]
The Poincar\'e inequality for functions with zero boundary mean then yields
\[
\|\nabla w_\beta\|_{L^2(P)}^2
\le
|P|^{1/2}\|w_\beta\|_{L^2(P)}
\le
C_P|P|^{1/2}\|\nabla w_\beta\|_{L^2(P)}.
\]
Hence
\[
\|w_\beta\|_{H^1(P)}\le C
\qquad\text{for all }\beta>0,
\]
where $C$ depends only on $P$. In particular, by the
Cauchy--Schwarz inequality,
\[\left|\int_P w_\beta\right|\le C.\]
It follows that
\[
\begin{aligned}
\T_\beta(P)
&=
\int_Pu_\beta
=
m_\beta|P|+\int_Pw_\beta\\
&=
\frac{|P|^2}{\beta L(P)}+O(1),
\end{aligned}
\]
which proves \eqref{eq:small-beta-torsion}.

Consequently, if $|P|=A$, then
\[
\T_\beta(R_{N,A})-\T_\beta(P)
=
\frac{A^2}{\beta}
\left(
\frac1{L(R_{N,A})}
-
\frac1{L(P)}
\right)
+O(1).
\]
For a nonregular tangential $N$-gon,
\[
L(P)>L(R_{N,A}),
\]
so the coefficient of $1/\beta$ is strictly positive. Hence
\[
\T_\beta(R_{N,A})-\T_\beta(P)
\asymp \frac1\beta
\qquad\text{as }\beta\downarrow0,
\]
in agreement with the order of the additional Robin term in
\eqref{eq:perimeter-deficit}.

	\section*{Statement on the use of AI tools}
	During the preparation of this manuscript, the authors used ChatGPT (OpenAI) for language and presentation assistance and as an interactive tool for exploratory mathematical discussion. These interactions served to explore and refine possible lines of argument. All mathematical statements, proofs, and computations in the final manuscript were independently verified by the authors, who take full responsibility for the article.


\end{document}